\def\DateTime{10/February/2010, 9:00 (Kyoto)}
\def\Version{Version $1.0$}
\def\yes{\if00}
\def\no{\if01}
\def\iftenpt{\no}
\def\ifelevenpt{\yes}
\def\iftwelvept{\no}

\def\ifusepdf{\no}
\def\ifpsfont{\yes}

\iftenpt
\documentclass[leqno]{amsart}
\else\fi
\ifelevenpt
\documentclass[leqno,11pt]{amsart}
\else\fi
\iftwelvept
\documentclass[leqno,12pt]{amsart}
\else\fi

\usepackage{amssymb}
\usepackage{amscd}
\usepackage{latexsym}
\usepackage{verbatim}
\usepackage{mathrsfs}
\usepackage[all]{xy}

\ifusepdf
\usepackage{hyperref}
\else\fi
\ifpsfont
\usepackage[T1]{fontenc}
\usepackage{times}
\else\fi

\ifelevenpt
\else\fi

\iftwelvept
\else\fi

\theoremstyle{plain}
\newtheorem{Theorem}{Theorem}[section]
\newtheorem{Proposition}[Theorem]{Proposition}
\newtheorem{Lemma}[Theorem]{Lemma}

\newtheorem{Claim}{Claim}[Theorem]

\theoremstyle{definition}

\newtheorem{Remark}[Theorem]{Remark}

\def\rom{\textup}
\newcommand{\ZZ}{{\mathbb{Z}}}
\newcommand{\QQ}{{\mathbb{Q}}}
\newcommand{\RR}{{\mathbb{R}}}

\newcommand{\CC}{{\mathbb{C}}}
\newcommand{\PP}{{\mathbb{P}}}

\newcommand{\Proj}{\operatorname{Proj}}

\newcommand{\adeg}{\widehat{\operatorname{deg}}}

\newcommand{\avol}{\widehat{\operatorname{vol}}}
\newcommand{\aH}{\hat{H}^0}

\newcommand{\Tpsh}{\operatorname{PSH}}

\newcommand{\rest}[2]{\left.{#1}\right\vert_{{#2}}}

\begin{document}

\title[Explicit computations of Zariski decompositions on $\PP^1_{\ZZ}$]%
{Explicit computations of Zariski decompositions on $\PP^1_{\ZZ}$}
\author{Atsushi Moriwaki}
\address{Department of Mathematics, Faculty of Science,
Kyoto University, Kyoto, 606-8502, Japan}
\email{moriwaki@math.kyoto-u.ac.jp}
\date{\DateTime, (\Version)}


\maketitle


\renewcommand{\theTheorem}{\Alph{Theorem}}
\section*{Introduction}

Let $\PP^1_{\ZZ} = \Proj(\ZZ[x, y])$, $C_0 = \{ x = 0 \}$, $C_{\infty} = \{ y = 0 \}$ and $z = x/y$.
For $a, b \in \RR_{>0}$, we define a $C_0$-Green function $g_{a,b}$ of ($\Tpsh \cap C^{\infty}$)-type on $\PP^1(\CC)$ and
an arithmetic divisor $\overline{D}_{a,b}$ of $(\Tpsh \cap C^{\infty})$-type on $\PP^1_{\ZZ}$
to be
\[
g_{a,b} := -\log \vert z \vert^2 + \log (a \vert z \vert^2 + b)\quad\text{and}\quad
\overline{D}_{a,b} := (C_0, g_{a,b}).
\]
In this note, we will observe several properties of $\overline{D}_{a,b}$ and 
give the exact form of the Zariski decomposition of $\overline{D}_{a,b}$.
The following results indicate us that $\overline{D}_{a,b}$ has
richer structure than what we expected.

\begin{Theorem}
Let $\varphi_{a,b} : [0, 1] \to \RR$ be a function given by
\[
\varphi_{a,b}(x) := (1-x)\log a + x \log b - x \log x - (1-x)\log(1-x) \quad(0 \leq x \leq 1),
\]
and let 
$\Theta_{a,b} := \{ x \in [0, 1] \mid \varphi_{a,b}(x) \geq 0 \}$.
Then the following properties hold for $\overline{D}_{a,b}$:

\begin{enumerate}
\renewcommand{\labelenumi}{(\arabic{enumi})}
\item
$\overline{D}_{a,b}$ is ample if and only if $a > 1$ and $b > 1$.

\item
$\overline{D}_{a,b}$ is nef if and only if $a \geq 1$ and $b \geq 1$.

\item
$\overline{D}_{a,b}$ is big if and only if $a + b > 1$.

\item
$\overline{D}_{a,b}$ is pseudo-effective if and only if $a + b \geq 1$.

{\nopagebreak
\begin{figure}[h]
\begin{center}
\unitlength=0.5mm
\begin{picture}(110,110)
\put(-5,0){\vector(1,0){100}}
\put(0,-5){\vector(0,1){100}}
\put(0,50){\line(1,-1){50}}
\put(50,50){\line(1,0){50}}
\put(50,50){\line(0,1){50}}
\put(97,-1){$a$}
\put(-1,97){$b$}
\put(13,15){\text{\tiny \rom{Not}}}
\put(2,10){\text{\tiny \rom{Pseudo-effective}}}
\put(70,70){\text{\tiny \rom{Ample}}}
\put(35,35){\text{\tiny \rom{Big}}}
\put(80,40){\vector(0,1){10}}
\put(60,35){\text{\tiny \rom{Nef on the boundary}}}
\put(40,20){\vector(-1,0){10}}
\put(41,18.5){\text{\tiny \rom{Pseudo-effective on the boundary}}}
\put(50,0){\circle*{2}}
\put(0,50){\circle*{2}}
\put(50,50){\circle*{2}}
\put(45,-5){\tiny \rom{(1,0)}}
\put(-11,47){\tiny \rom{(0,1)}}
\put(45,45){\tiny \rom{(1,1)}}
\end{picture}
\caption{Geography of $\overline{D}_{a,b}$}
\end{center}
\end{figure}
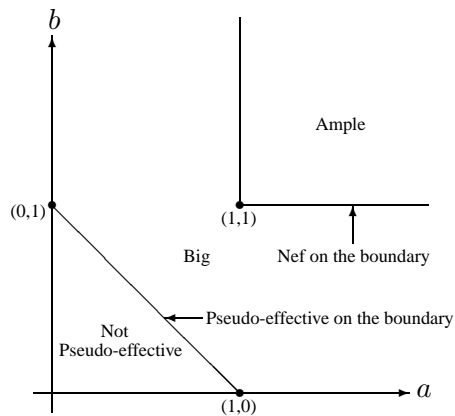}

\item
$\aH(\PP^1_{\ZZ}, n \overline{D}_{a,b}) \not= \{ 0 \}$ if and only if
$n \Theta_{a,b} \cap \ZZ \not= \emptyset$. As consequences, we have the following:

\begin{enumerate}
\renewcommand{\labelenumii}{(\arabic{enumi}.\arabic{enumii})}
\item
We assume that $a + b = 1$. For a positive integer $n$,
\[
\aH(\PP^1_{\ZZ}, n \overline{D}_{a,b}) = \begin{cases}
\{ 0, \pm z^{-nb} \} & \text{if $nb \in \ZZ$},\\
\{ 0 \} & \text{if $nb \not\in \ZZ$}.
\end{cases}
\]
In particular, if $b \not\in \QQ$, then $\aH(\PP^1_{\ZZ}, n \overline{D}_{a, b}) = \{ 0 \}$ for all $n \geq 1$.

\item
For any positive integer $n$,
there exist $a, b \in \RR_{>0}$ such that
$\overline{D}_{a,b}$ is big and 
\[
\aH(\PP^1_{\ZZ}, l \overline{D}_{a,b} ) = \{ 0 \}
\]
for all $l$ with
$1 \leq l \leq n$.
\end{enumerate}

\item
${\displaystyle \left\langle \aH(\PP^1_{\ZZ}, n \overline{D}_{a,b}) \right\rangle_{\ZZ} = \bigoplus_{i \in n \Theta_{a,b} \cap \ZZ} \ZZ z^{-i}}$
if $n \Theta_{a,b} \cap \ZZ \not= \emptyset$.

\item
\rom{(}Integral formula\rom{)}\quad
${\displaystyle \avol(\overline{D}_{a,b}) = \int_{\Theta_{a,b}} \varphi_{a, b}(x) dx}$.

\item
\rom{(}Zariski decomposition\rom{)}\quad
The Zariski decomposition exists if and only if $a + b \geq 1$.
Moreover, the positive part of $\overline{D}_{a,b}$ is given by $(\theta_{a,b}C_0 - \vartheta_{a,b}C_{\infty}, p_{a,b})$,
where $\vartheta_{a,b} = \inf \Theta_{a,b}$, $\theta_{a,b} = \sup \Theta_{a,b}$ and
\[
\hspace{4em}
p_{a,b}(z) =\begin{cases}
-\theta_{a,b} \log \vert z \vert^2 & \text{if $\vert z \vert \leq \sqrt{\frac{b(1-\theta_{a,b})}{a\theta_{a,b}}}$}, \\
-\log \vert z \vert^2 + \log (a \vert z \vert^2 + b) & 
\text{if $\sqrt{\frac{b(1-\theta_{a,b})}{a\theta_{a,b}}} < \vert z \vert < \sqrt{\frac{b(1-\vartheta_{a,b})}{a\vartheta_{a,b}}}$}, \\
-\vartheta_{a,b} \log \vert z \vert^2 & \text{if $\vert z \vert \geq \sqrt{\frac{b(1-\vartheta_{a,b})}{a\vartheta_{a,b}}}$}.
\end{cases}
\]
In particular, if $a + b = 1$, then the positive part is $b\widehat{(z)}$.
\end{enumerate}
\end{Theorem}

I would like to express my thanks to Prof. Yuan.
The studies of this note  started from his question.
I also thank  Dr. Uchida.
Without his calculation of the limit of a sequence,
I could not find the positive part of $\overline{D}_{a,b}$.

\renewcommand{\theTheorem}{\arabic{section}.\arabic{Theorem}}
\renewcommand{\theClaim}{\arabic{section}.\arabic{Theorem}.\arabic{Claim}}
\renewcommand{\theequation}{\arabic{section}.\arabic{Theorem}.\arabic{Claim}}

\section{Fundamental properties of the characteristic function}
\label{sec:fund:prop:func}
Let $\PP^1_{\ZZ} = \Proj(\ZZ[x, y])$, $C_0 = \{ x = 0 \}$, $C_{\infty} = \{ y = 0 \}$ and $z = x/y$.
Let us fix positive real numbers $a$ and $b$.
We set 
\[
g_{a,b} = -\log \vert z \vert^2 + \log (a \vert z \vert^2 + b)\quad\text{and}\quad
\Phi_{a,b} = dd^c(\log (a \vert z \vert^2 + b))
\]
on $\PP^1(\CC)$.
It is easy to see that $g_{a,b}$ is a $C_0$-Green function of $(\Tpsh \cap C^{\infty})$-type and
\[
\Phi_{a,b} = \frac{ab}{2\pi\sqrt{-1}( a \vert z \vert^2 + b)^2}dz \wedge d\bar{z}.
\]
Note that $H^0(\PP^1_{\ZZ}, nC_0) = \bigoplus_{0 \leq i \leq n} \ZZ z^{-i}$.
According as \cite{MoArZariski}, $\vert\cdot\vert_{ng_{a,b}}$, $\Vert\cdot\Vert_{ng_{a,b}}$ and
$\langle\cdot,\cdot\rangle_{ng_{a,b}}$ are defined by
\begin{align*}
& \vert\phi\vert_{ng_{a,b}} := \vert \phi \vert \exp(-ng_{a,b}/2),\quad
\Vert\phi\Vert_{ng_{a,b}} := \sup \{ \vert\phi\vert_{ng_{a,b}}(x) \mid x \in \PP^1(\CC) \},
\intertext{and}
& \langle\phi,\psi \rangle_{ng_{a,b}} := \int_{\PP^1(\CC)} \phi \bar{\psi} \exp(-ng_{a,b}) \Phi_{a,b},
\end{align*}
where $\phi, \psi \in H^0(\PP^1(\CC), nC_0)$.
Moreover, we define a function $\varphi_{a,b} : [0, 1] \to \RR$ to be
\[
\varphi_{a,b}(x) = (1-x)\log a + x \log b - x \log x - (1-x)\log(1-x)  \quad(0 \leq x \leq 1),
\]
which is called the {\em characteristic function of $g_{a,b}$}. 
The function $\varphi_{a,b}$ play a key role in this note.
Here note that
\[
\begin{cases}
\varphi_{a,b}(0) = \log a,\\
\varphi_{a,b}(1) = \log b, \\
\max \{ \varphi_{a,b}(x) \mid x \in [0,1] \} = \varphi_{a,b}(b/(a+b)) = \log(a+b).
\end{cases}
\]
Moreover, $\varphi_{a,b}$ is strictly increasing on $[0, b/(a+b)]$ and
$\varphi_{a,b}$ is strictly decreasing on $[b/(a+b),1]$.
Notably the function $\varphi_{1,1}$ is very similar to the binary entropy function
$H(x) = - x \log_2 x - (1-x)\log_2(1-x)$.

\begin{Proposition}
\label{prop:cal:inner:product}
For $i, j, n \in \ZZ$ with $0 \leq i, j \leq n$, we have the following:
\begin{enumerate}
\renewcommand{\labelenumi}{(\arabic{enumi})}
\item
$\Vert z^{-i} \Vert_{ng_{a,b}}^2 = \exp(-n \varphi_{a,b}(i/n))$.

\item
\[
\langle z^{-i}, z^{-j} \rangle_{ng_{a,b}} = \begin{cases}
0 & \text{if $i \not= j$}, \\
\\
{\displaystyle \frac{1}{(n+1)\binom{n}{i}a^{n-i}b^i}} & \text{if $i=j$}.
\end{cases}
\]
\end{enumerate}
\end{Proposition}

\begin{proof}
(1) By the definition of $\vert z^{-i} \vert_{ng_{a,b}}$, we can see
\[
\log \vert z^{-i} \vert_{ng_{a,b}}^2 = (n-i) \log \vert z \vert^2 - n \log(a\vert z \vert^2 + b).
\]
If we set $f(x) = (n-i) \log x - n \log(ax + b)$ for $ x \geq 0$ ($f(0)$ might be $-\infty$),
then
\[
\max \{ f(x) \mid x \geq 0 \} = f\left(\frac{(n-i)b}{ia}\right) = -n \varphi_{a,b}(i/n).
\]
Thus (1) follows.

\medskip
(2) First of all,
\[
\langle z^{-i}, z^{-j} \rangle_{ng_{a,b}} = \frac{ab}{2\pi\sqrt{-1}} \int_{\PP^1} \frac{z^{n-i} \bar{z}^{n-j}}{(a \vert z \vert^2 + b)^{n+2}} d z \wedge d\bar{z}.
\]
If we set $z = r^{1/2} \exp(\sqrt{-1}t)$, then the above integral is equal to
\[
\frac{ab}{2\pi} \int_{0}^{\infty} \frac{r^{n - (i+j)/2}}{(ar + b)^{n+2}} \left( 
\int_{0}^{2\pi} \exp((j-i)\sqrt{-1}t)  dt \right) dr,
\]
and hence
\[
\langle z^{-i}, z^{-j} \rangle_{ng_{a,b}} =
\begin{cases}
0 & \text{if $i \not= j$}, \\
\\
{\displaystyle ab\int_{0}^{\infty} \frac{r^{(n - i)}}{(ar + b)^{n+2}} dr} & \text{if $i=j$}.
\end{cases}
\]
For $k, l \in \ZZ$ with $0 \leq k \leq l$, we denote
\[
ab\int_{0}^{\infty} \frac{r^{(l - k)}}{(ar + b)^{l+2}} dr
\]
by $I(k, l)$.
It is easy to see that
\[
I(l,l) = \frac{1}{(l+1)b^l}\quad\text{and}\quad
I(k, l) = \frac{l-k}{a(l+1)} I(k, l-1)\ \ (0 \leq k < l).
\]
Thus the proposition follows.
\end{proof}

Next we observe the following lemma:

\begin{Lemma}
\label{lem:limit:binom}
Let $\{ n_l \}_{l=1}^{\infty}$ and $\{ c_l \}_{l=1}^{\infty}$ be sequences of integers satisfying the following properties:
\begin{enumerate}
\renewcommand{\labelenumi}{(\alph{enumi})}
\item $0 < n_1 < n_2 < \cdots < n_{l} < n_{l+1} < \cdots$.

\item $0 \leq c_l \leq n_l$ for all $l \geq 1$.

\item $\gamma = \lim_{l\to\infty} c_l/n_l$ exists.
\end{enumerate}
Let $\{ d_l \}$ be a sequence of positive real numbers such that $\lim_{l\to\infty} \log(d_l)/(n_l+1) = 0$.
Then we have the following:
\begin{enumerate}
\renewcommand{\labelenumi}{(\arabic{enumi})}
\item
If ${\displaystyle \binom{n_l}{c_l}a^{n_l-c_l}b^{c_l} \geq d_l}$ for all $l \geq 1$, then $\varphi_{a,b}(\gamma) \geq 0$.

\item
If ${\displaystyle \binom{n_l}{c_l}a^{n_l-c_l}b^{c_l} \leq d_l}$ for all $l \geq 1$, then $\varphi_{a,b}(\gamma) \leq 0$.
\end{enumerate}
\end{Lemma}

\begin{proof}
First of all, let us see the following claim:
\begin{Claim}
\label{claim:lem:limit:binom}
For $0 \leq i \leq n$,
\[
\int_{1/(n+1)}^{(i+1)/(n+1)} \log\left( \frac{1}{t} - 1\right) dt \leq \frac{1}{n+1} \log \binom{n}{i} \leq \int_0^{i/(n+1)} \log\left( \frac{1}{t} - 1\right) dt.
\]
\end{Claim}

\begin{proof}
Indeed, if $i=0$, then the assertion is obvious. If $i \geq 1$,
then
\[
\frac{1}{n+1} \log \binom{n}{i}  = \frac{1}{n+1} 
\sum_{k=1}^{i} \log \left( \frac{1}{\frac{k}{n+1}} - 1 \right).
\]
In addition, since $\log(1/x - 1)$ is strictly decreasing,
\[
\int_{k/(n+1)}^{(k+1)/(n+1)} \log\left( \frac{1}{t} - 1\right) dt  \leq \frac{1}{n+1} \log \left( \frac{1}{\frac{k}{n+1}} - 1 \right) \leq \int_{(k-1)/(n+1)}^{k/(n+1)} \log\left( \frac{1}{t} - 1\right) dt.
\]
Therefore the claim follows.
\end{proof}

Note that
\[
\varphi_{a,b}(x) = (1-x)\log a + x \log b + \int_{0}^x  \log\left( \frac{1}{t} - 1\right) dt.
\]
Thus the above estimate implies that
\addtocounter{Claim}{1}
\begin{multline}
\label{lem:limit:binom:eqn:1}
\varphi_{a,b}((i+1)/(n+1)) - \varphi_{1,1}(1/(n+1)) - \frac{1}{n+1} \log b \\
\leq \frac{1}{n+1} \log \left(\binom{n}{i}a^{n-i}b^i \right) \leq \varphi_{a,b}(i/(n+1)) - \frac{1}{n+1} \log a,
\end{multline}
and hence the lemma follows.
\end{proof}

Let $\overline{D}_{a,b}$ be an arithmetic divisor of ($\Tpsh \cap C^{\infty}$)-type on $\PP^1_{\ZZ}$ 
given by
\[
\overline{D}_{a, b} = (C_0, g_{a,b}) = (C_0, -\log\vert z \vert^2 + \log (a \vert z \vert^2 + b)).
\]
Moreover, we set
\[
\Theta_{a,b} = \{ x \in [0,1] \mid \varphi_{a,b}(x) \geq 0 \}.
\]
Note that if $\Theta_{a,b} \not= \emptyset$, then
\[
\Theta_{a,b} = \{ x \in [0,1] \mid \inf \Theta_{a,b} \leq x \leq 
\sup \Theta_{a,b} \}.
\]
Finally we consider the following proposition:

\begin{Proposition}
\label{prop:criterion:QQ:effective}
Let us fix a positive integer $n$. Then we have the following:
\begin{enumerate}
\renewcommand{\labelenumi}{(\arabic{enumi})}
\item
$\aH(\PP^1_{\ZZ}, n \overline{D}_{a,b}) \not= \{ 0 \}$ if and only if
$n \Theta_{a,b} \cap \ZZ \not= \emptyset$.

\item
If $n \Theta_{a,b} \cap \ZZ \not= \emptyset$,
then $\langle \aH(\PP^1_{\ZZ}, n \overline{D}_{a,b}) \rangle_{\ZZ} = \bigoplus_{i \in n \Theta_{a,b} \cap \ZZ} \ZZ z^{-i}$.
\end{enumerate}
\end{Proposition}

\begin{proof}
Let us begin with the following claim:

\begin{Claim}
If $\phi \in \aH(\PP^1_{\ZZ}, n \overline{D}_{a,b}) \setminus \{ 0 \}$,
then $n \Theta_{a,b} \cap \ZZ \not= \emptyset$ and
$\phi \in \bigoplus_{i \in n \Theta_{a,b} \cap \ZZ} \ZZ z^{-i}$.
\end{Claim}

\begin{proof}
We set $\phi = \sum_{i=0}^n c_i z^{-i}$, where $c_0, \ldots, c_n \in \ZZ$.
Let $k_1 = \min \{ i \mid c_i \not= 0 \}$ and $k_2 = \max \{ i \mid c_i \not= 0 \}$.
Then $\phi = \sum_{k_1 \leq i \leq k_2} c_i z^{-i}$.
Note that $\phi^l \in \aH(\PP^1_{\ZZ}, ln\overline{D}_{a,b})$, so that $\langle \phi^l, \phi^l \rangle_{lng_{a,b}} \leq 1$.
Moreover, 
\[
\phi^l = (c_{k_1})^l z^{-k_1l} + (c_{k_2})^l z^{-k_2l} + \sum_{k_1l < j < k_2l} c_{l, j} z^{-j}
\]
and
\begin{multline*}
\langle \phi^l, \phi^l \rangle_{lng_{a,b}}  = \frac{(c_{k_1})^{2l}}{(nl+1)\binom{nl}{k_1l}a^{nl-k_1l}b^{k_1l}} + 
\frac{(c_{k_2})^{2l}}{(nl+1)\binom{nl}{k_2l}a^{nl-k_2l}b^{k_2l}} \\
+
\sum_{k_1l < j < k_2l} \frac{(c_{l,j})^2}{(nl+1)\binom{nl}{j}a^{nl-j}b^j}
\end{multline*}
by Proposition~\ref{prop:cal:inner:product}.
Therefore, 
\[
(nl+1)\binom{nl}{k_1l}a^{nl-k_1l}b^{k_1l} \geq 1\quad\text{and}\quad (nl+1)\binom{nl}{k_2l}a^{nl-k_2l}b^{k_2l} \geq 1
\]
for all $l \geq 1$, and hence
$\varphi_{a,b}(k_1/n) \geq 0$ and $\varphi_{a,b}(k_2/n) \geq 0$ by Lemma~\ref{lem:limit:binom}.
Therefore, $k_1, k_2 \in n \Theta_{a,b} \cap \ZZ$, which yields $\{ i \in \ZZ \mid k_1 \leq i \leq k_2 \} \subseteq n\Theta_{a,b} \cap \ZZ$,
as required.
\end{proof}

Note that $\Vert z^{-i} \Vert_{ng} =  \exp(-n\varphi_{a,b}(i/n))$ by Proposition~\ref{prop:cal:inner:product}.
Thus
(1) and (2) follow from the above claim.
\end{proof}

\section{Integral formula and Geography of $\overline{D}_{a,b}$}
Let $X$ be a $d$-dimensional generically smooth, normal and projective arithmetic variety.
Let $\overline{D} = (D, g)$ be an arithmetic $\RR$-divisor of $C^0$-type on $X$.
Let $\Phi$ be an $F_{\infty}$-invariant volume form on $X(\CC)$ with ${\displaystyle \int_{X(\CC)} \Phi = 1}$.
For $\phi, \psi \in H^0(X, D)$, $\langle \phi, \psi \rangle_g$ and $\Vert \phi \Vert_{g, L^2}$ are given by
\[
\langle \phi, \psi \rangle_g := \int_{X(\CC)} \phi \bar{\psi} \exp(-g) \Phi\quad\text{and}\quad
\Vert \phi \Vert_{g, L^2} := \sqrt{\langle \phi, \phi \rangle_g}.
\]
We set
\[
\aH_{L^2}(X, \overline{D}) := \{ \phi \in H^0(X, D) \mid \Vert \phi\Vert_{g, L^2} \leq 1 \}.
\]
Let us begin with the following lemma:

\begin{Lemma}
\label{lem:vol:L:2}
${\displaystyle \avol(\overline{D}) = \lim_{n\to\infty} \frac{\log \# \aH_{L^2}(X, n\overline{D})}{n^d/d!}}$.
\end{Lemma}

\begin{proof}
Since $\aH(X, n\overline{D}) \subseteq \aH_{L^2}(X, n\overline{D})$, we have
\[
\avol(\overline{D}) \leq \liminf_{n\to\infty} \frac{\log \# \aH_{L^2}(X, n\overline{D})}{n^d/d!}.
\]
On the other hand, by using Gromov's inequality (cf. \cite[Proposition~3.1.1]{MoArZariski}), 
there is a constant $C$ such that $\Vert\cdot\Vert_{\sup} \leq Cn^{d-1}\Vert\cdot\Vert_{L^2}$
on $H^0(X, nD)$.
Thus, for any positive number $\epsilon$,
$\Vert\cdot\Vert_{\sup} \leq \exp(n\epsilon/2) \Vert\cdot\Vert_{L^2}$ holds for $n \gg 1$.
This implies that 
\[
\aH_{L^2}(X, n\overline{D}) \subseteq \aH(X, n(\overline{D} + (0, \epsilon)))
\]
for $n \gg 1$, which yields
\[
 \limsup_{n\to\infty} \frac{\log \# \aH_{L^2}(X, n\overline{D})}{n^d/d!} \leq \avol(\overline{D} + (0, \epsilon)).
\]
Therefore, by virtue of the continuity of $\avol$, we have
\[
 \limsup_{n\to\infty} \frac{\log \#\aH_{L^2}(X, n\overline{D})}{n^d/d!} \leq \avol(\overline{D}),
\]
and hence the lemma follows.
\end{proof}

From now on,
we use the same notation as in Section~\ref{sec:fund:prop:func}.
The purpose of this section is to prove the following theorem:

\begin{Theorem}
\label{thm:positivity:D:a:b}
\begin{enumerate}
\renewcommand{\labelenumi}{(\arabic{enumi})}
\item \rom{(}Integral formula\rom{)}
\[
\avol(\overline{D}_{a,b}) = \int_{\Theta_{a,b}} \varphi_{a,b}(x) dx.
\]

\item
$\overline{D}_{a,b}$ is ample if and only if $a > 1$ and $b > 1$.

\item
$\overline{D}_{a,b}$ is nef if and only if $a \geq 1$ and $b \geq 1$.

\item
$\overline{D}_{a,b}$ is big if and only if $a + b > 1$.

\item
$\overline{D}_{a,b}$ is pseudo-effective if and only if $a + b \geq 1$.

\item
If $a + b = 1$, then
\[
\aH(\PP^1_{\ZZ}, n\overline{D}_{a,b}) =
\begin{cases}
\{ 0, \pm z^{-nb} \} & \text{if $nb \in \ZZ$}, \\
\{ 0 \} & \text{if $nb \not\in \ZZ$}.
\end{cases}
\]  
\end{enumerate}
\end{Theorem}

\begin{proof}
First let us see the essential case of (1):

\begin{Claim}
\label{claim:thm:positivity:D:a:b:1}
If $a + b > 1$, then ${\displaystyle \avol(\overline{D}_{a,b}) = \int_{\Theta_{a,b}} \varphi_{a,b}(x) dx}$
\end{Claim}

\begin{proof}
In this case, since $\operatorname{vol}(\Theta_{a,b}) > 0$,
we can find a positive integer $n_0$ such that $n  \Theta_{a,b} \cap \ZZ \not = \emptyset$ for all $n \geq n_0$.
Let $\vartheta_n$ and $\theta_n$ be $\min (n  \Theta_{a,b} \cap \ZZ )$ and $\max (n  \Theta_{a,b} \cap \ZZ )$ respectively.
Then, by Proposition~\ref{prop:criterion:QQ:effective},
\[
\aH(\PP^1_{\ZZ}, n\overline{D}_{a,b}) \subseteq \left.\left\{ \phi \in \bigoplus_{i = \vartheta_n}^{\theta_n} \ZZ z^{-i} \ \right| \ 
\langle \phi, \phi \rangle_{ng_{a,b}} \leq 1 \right\}
\subseteq  \aH_{L^2}(\PP^1_{\ZZ}, n\overline{D}_{a,b}),
\]
which yields
\[
\avol(\overline{D}_{a,b}) = \lim_{n\to\infty} \frac{2\log \#\left.\left\{ \phi \in \bigoplus_{i = \vartheta_n}^{\theta_n} \ZZ z^{-i} \ \right| \ 
\langle \phi, \phi \rangle_{ng_{a,b}} \leq 1 \right\}}{(n+1)^2}
\]
by Lemma~\ref{lem:vol:L:2}.
We set
\[
K_n = \left\{ (x_{\vartheta_{n}}, \ldots, x_{\theta_{n}}) \in \RR^{\theta_{n} - \vartheta_{n} + 1}\ \left| \  
\sum_{i=\vartheta_{n}}^{\theta_{n}} \frac{x_i^2}{(n+1)\binom{n}{i}a^{n-i}b^i} \leq 1 \right\}\right. .
\]
Then, by Proposition~\ref{prop:cal:inner:product},
\[
\#  \left.\left\{ \phi \in \bigoplus_{i = \vartheta_n}^{\theta_n} \ZZ z^{-i} \ \right| \ 
\langle \phi, \phi \rangle_{ng_{a,b}} \leq 1 \right\}
= \# (K_n \cap \ZZ^{\theta_{n} - \vartheta_{n} + 1}).
\]
By Minkowski's theorem,
\[
\log \# (K_n \cap \ZZ^{\theta_{n} - \vartheta_{n} + 1}) \geq \log(\operatorname{vol}(K_n)) - (\theta_{n} - \vartheta_{n} + 1) \log(2).
\]
Note that
\[
 \log(\operatorname{vol}(K_n)) = \sum_{i=\vartheta_{n}}^{\theta_{n}}\log \left( \sqrt{(n+1)\binom{n}{i}a^{n-i}b^i}\right)
+ \log V_{\theta_{n} - \vartheta_{n} + 1},
\]
where $V_r = \operatorname{vol}( \{(x_{1}, \ldots, x_{r}) \in \RR^{r} \mid 
x_1^2 + \cdots + x_r^2 \leq 1\})$.
Thus
\begin{multline*}
\log \# (K_n \cap \ZZ^{\theta_{n} - \vartheta_{n} + 1}) \\
\geq  \sum_{i=\vartheta_{n}}^{\theta_{n}}\log \left( \sqrt{(n+1)\binom{n}{i}a^{n-i}b^i}\right)
+ \log V_{\theta_{n} - \vartheta_{n} + 1} - (\theta_{n} - \vartheta_{n} + 1) \log(2).
\end{multline*}
On the other hand, since
\[
K_n  \subseteq  \prod_{\vartheta_{n} \leq i \leq \theta_{n}} \left[ -\sqrt{(n+1)\binom{n}{i}a^{n-i}b^i},
\sqrt{(n+1)\binom{n}{i}a^{n-i}b^i} \right],
\]
we have
\[
\log \# (K_n \cap \ZZ^{\theta_{n} - \vartheta_{n} + 1})
\leq \sum_{i=\vartheta_{n}}^{\theta_{n}}\log \left( 2\sqrt{(n+1)\binom{n}{i}a^{n-i}b^i} + 1\right).
\]
Moreover, by Proposition~1.1, for $\vartheta_{a,b} \leq i \leq \theta_{a,b}$,
\[
(n+1)\binom{n}{i}a^{n-i}b^i = \frac{1}{\langle z^{-i}, z^{-i} \rangle_{ng_{a,b}}} \geq \exp(n\varphi_{a,b}(i/n)) \geq 1.
\]
Thus
\[
\log \# (K_n \cap \ZZ^{\theta_{n} - \vartheta_{n} + 1})
\leq \sum_{i=\vartheta_{n}}^{\theta_{n}}\log \left( \sqrt{(n+1)\binom{n}{i}a^{n-i}b^i}\right)  + (\theta_{n} - \vartheta_{n} + 1) \log(3)
\]
because $2 t + 1 \leq 3t$ for $t \geq 1$.
Therefore, it is sufficient to show that
\[
\lim_{n\to\infty} \frac{1}{(n+1)^2} \sum_{i=\vartheta_{n}}^{\theta_{n}}\log \left(\binom{n}{i}a^{n-i}b^i \right) =  \int_{\Theta_{a,b}} \varphi_{a,b}(x) dx
\]
because $\lim_{n\to\infty} \log (V_{\theta_{n} - \vartheta_{n} + 1})/n^2 = 0$.
By the estimate \eqref{lem:limit:binom:eqn:1}, we have
\begin{multline*}
\lim_{n\to\infty} \frac{1}{n+1} \sum_{i=\vartheta_{n}}^{\theta_{n}} \varphi_{a,b}\left(\frac{i+1}{n+1}\right) \\
\leq 
\lim_{n\to\infty} \frac{1}{(n+1)^2} \sum_{i=\vartheta_{n}}^{\theta_{n}}\log \left(\binom{n}{i}a^{n-i}b^i \right) \\
\leq \lim_{n\to\infty} \frac{1}{n+1} \sum_{i=\vartheta_{n}}^{\theta_{n}} \varphi_{a,b}\left(\frac{i}{n+1}\right).
\end{multline*}
Thus the claim follows because $\lim_{n\to\infty} \vartheta_n/n = \inf \Theta_{a,b}$ and $\lim_{n\to\infty} \theta_n/n = \sup \Theta_{a,b}$.
\end{proof}

Next let us see the following claim:
\begin{Claim}
\label{claim:thm:positivity:D:a:b:2}
If $s, t \in \RR_{>0}$ and $\alpha, \beta \in \RR$ with $\alpha + \beta \not= 0$,
then
\[
\alpha \overline{D}_{ta,tb} + \beta \overline{D}_{sa,  sb} = (\alpha + \beta)\overline{D}_{(t^{\alpha}s^{\beta})^{\frac{1}{\alpha + \beta}}a, \ (t^{\alpha}s^{\beta})^{\frac{1}{\alpha + \beta}}b}.
\]
\end{Claim}

\begin{proof}
This is a straightforward calculation.
\end{proof}

\medskip
(2) and (3):\quad
First of all, note that $\adeg(\rest{\overline{D}_{a,b}}{C_0}) = \log(b)$ and
$\adeg(\rest{\overline{D}_{a,b}}{C_{\infty}}) = \log(a)$.
Moreover, $\log(a+bx) \geq \log(a)$ for all $x \in \RR_{\geq 0}$.
Thus $\overline{D}_{a,b}$ is nef if and only if $a \geq 1$ and $b \geq 1$.
Further, $\adeg(\overline{D}_{a,b}^2) = (\log(ab) + 1)/2$.
Therefore, by the arithmetic Nakai-Moishezon's criterion,
$\overline{D}_{a,b}$ is ample if and only if $a > 1$ and $b > 1$.

\medskip
(6):\quad
In this case, $\Theta_{a, b} = \{ b \}$.
Moreover, if $nb \in \ZZ$, then
\[
\Vert z^{-nb} \Vert_{ng_{a,b}}^2 = \exp(-n \varphi_{a,b}(b)) = 1
\]
by Proposition~\ref{prop:cal:inner:product}.
Thus the assertion follows from Proposition~\ref{prop:criterion:QQ:effective}.

\medskip
(4) and (5):\quad
By using (6), in order to see (4) and (5), it is sufficient to show the following:

\begin{enumerate}
\renewcommand{\labelenumi}{(\roman{enumi})}
\item
$D_{a,b}$ is big if $a + b > 1$.

\item
$D_{a,b}$ is pseudo-effective if $a + b \geq 1$.

\item
$D_{a,b}$ is not pseudo-effective if $a + b < 1$.
\end{enumerate}

\medskip
(i) It follows from Claim~\ref{claim:thm:positivity:D:a:b:1} because $\operatorname{vol}(\Theta_{a,b}) > 0$.

\medskip
(ii) We choose a real number $t$ such that $t > 1$ and $\overline{D}_{ta,tb}$ is ample.
By Claim~\ref{claim:thm:positivity:D:a:b:2},
\[
\overline{D}_{a,b} + \epsilon \overline{D}_{ta, tb} = (1 + \epsilon) \overline{D}_{t^{\frac{\epsilon}{1 + \epsilon}}a, t^{\frac{\epsilon}{1 + \epsilon}}b}.
\]
For any $\epsilon > 0$,
since $t^{\frac{\epsilon}{1 + \epsilon}}(a + b) > 1$, 
$(1 + \epsilon) \overline{D}_{t^{\frac{\epsilon}{1 + \epsilon}}a, t^{\frac{\epsilon}{1 + \epsilon}}b}$ is big by (i),
which shows that $\overline{D}_{a,b}$ is pseudo-effective.

\medskip
(iii) 
Let us choose a positive real number $t$ such that $\overline{D}_{ta,tb}$ is ample.
We also choose a positive number $\epsilon$ such that if we set $a' = a t^{\frac{\epsilon}{1 + \epsilon}}$ and $b' = b t^{\frac{\epsilon}{1 + \epsilon}}$,
then $a' + b' < 1$.
We assume that $\overline{D}_{a,b}$ is pseudo-effective. Then 
\[
\overline{D}_{a,b} + \epsilon \overline{D}_{ta, tb} = (1 + \epsilon) \overline{D}_{a', b'}
\]
is big
by \cite[Proposition~6.3.2]{MoArZariski}, which means that $\overline{D}_{a', b'}$ is big.
On the other hand, as $a' + b' < 1$, we have
$\Theta_{a', b'} = \emptyset$.
Thus $\aH(\PP^1_{\ZZ}, n\overline{D}_{a', b'}) = \{ 0 \}$
for all $n \geq 1$ by Proposition~\ref{prop:criterion:QQ:effective}. This is a contradiction.

\bigskip
Finally let us see (1).
By Claim~\ref{claim:thm:positivity:D:a:b:1}, we may assume that $a + b \leq 1$.
In this case,
$\overline{D}_{a,b}$ is not big by (4) and
$\Theta_{a,b}$ is either $\emptyset$ or $\{ b \}$.
Thus the assertion follows.
\end{proof}

\begin{Remark}
By a straightforward calculation, we can see
\[
\adeg(\overline{D}_{a,b}^2) = \int_{[0,1]} \varphi_{a,b}(x) dx = (\log(ab) + 1)/2.
\]
In particular, $\adeg(\overline{D}_{a,b}^2) = \avol(\overline{D}_{a,b})$ if and only if
$\overline{D}_{a,b}$ is nef.
\end{Remark}

\bigskip
Finally let us consider the following proposition:

\begin{Proposition}
For any positive integer $n$, there exist rational numbers $a$ and $b$ such that
$0 < a < 1$, $0 < b < 1$, $a + b > 1$ and that
$\aH(\PP^1_{\ZZ}, l \overline{D}_{a,b}) = \{ 0 \}$ for $l=1, \ldots, n$.
\end{Proposition}

\begin{proof}
Let us choose rational numbers $a'$ and $b'$ such that
$0 < a' < 1$, $0 < b' < 1/n$ and $a' + b' = 1$.
Since $\varphi_{a',b'}(1/n) < 0$, we can find
a rational number $\lambda$ such that
$\lambda > 1$, $\lambda a' < 1$, $\lambda b' < 1$ and $\varphi_{a',b'}(1/n) + \log(\lambda) < 0$.
Here we set $a = \lambda a'$ and $b = \lambda b'$. Then we have
\[
\Theta_{a,b} \subseteq \{ x \in \RR \mid 0 < x < 1/n \}
\]
because
\[
\begin{cases}
\varphi_{a,b}(0) = \log(a) < 0,\\
\varphi_{a,b}(1/n) = \varphi_{a', b'}(1/n) + \log(\lambda) < 0, \\
b/(a+b) = b' < 1/n.
\end{cases}
\]
We assume that $\aH(\PP^1_{\ZZ}, l \overline{D}_{a,b}) \not= \{ 0 \}$ for some $l$ with $1 \leq l \leq n$.
Then, by Proposition~\ref{prop:criterion:QQ:effective}, there is an integer $k$ such that $0 \leq k \leq l$ and $\varphi_{a,b}(k/l) \geq 0$.
Thus $k/l \in \Theta_{a,b}$, that is, $0 < k/l < 1/n$, which is a contradiction.
\end{proof}

\section{Zariski decomposition of $\overline{D}_{a,b}$}
We use the same notation as in Section~\ref{sec:fund:prop:func}.
In this section, let us consider the Zariski decomposition of $\overline{D}_{a,b}$.

\begin{Theorem}
\label{thm:zariski:decomp:PP:1}
The Zariski decomposition exists if and only if $a + b \geq 1$.
Moreover, if we set 
$\vartheta_{a,b} = \inf \Theta_{a,b}$, $\theta_{a,b} = \sup \Theta_{a,b}$,
$P_{a,b} = \theta_{a,b}C_0 - \vartheta_{a,b}C_{\infty}$ and
\[
p_{a,b}(z) =\begin{cases}
-\theta_{a,b} \log \vert z \vert^2 & \text{if $\vert z \vert \leq \sqrt{\frac{b(1-\theta_{a,b})}{a\theta_{a,b}}}$}, \\
-\log \vert z \vert^2 + \log (a \vert z \vert^2 + b) & 
\text{if $\sqrt{\frac{b(1-\theta_{a,b})}{a\theta_{a,b}}} < \vert z \vert < \sqrt{\frac{b(1-\vartheta_{a,b})}{a\vartheta_{a,b}}}$}, \\
-\vartheta_{a,b} \log \vert z \vert^2 & \text{if $\vert z \vert \geq \sqrt{\frac{b(1-\vartheta_{a,b})}{a\vartheta_{a,b}}}$},
\end{cases}
\]
then the positive part of $\overline{D}_{a,b}$ is $\overline{P}_{a,b} = (P_{a,b}, p_{a,b})$.
\end{Theorem}

\begin{proof}
First we consider the case where $\overline{D}_{a,b}$ is big, that is, $a + b > 1$ by Theorem~\ref{thm:positivity:D:a:b}.
In this case, $0 \leq \vartheta_{a,b} < \theta_{a,b} \leq 1$.
The existence of the Zariski decomposition follows from \cite[Theorem~9.2.1]{MoArZariski}.
Here we consider functions 
\begin{align*}
& r_1 : \left\{ z \in \CC \mid \vert z \vert < \sqrt{\frac{b(1-\vartheta_{a,b})}{a\vartheta_{a,b}}} \right\} \to \RR \\
\intertext{and}
& r_2 : \left\{ z \in \PP^1(\CC) \mid \vert z \vert > \sqrt{\frac{b(1-\theta_{a,b})}{a\theta_{a,b}}} \right\} \to \RR
\end{align*}
given by
\[
r_1(z) = \begin{cases}
0  & \text{if $\vert z \vert < \sqrt{\frac{b(1-\theta_{a,b})}{a\theta_{a,b}}}$}, \\
-(1 - \theta_{a,b}) \log \vert z \vert^2 + \log (a \vert z \vert^2 + b) & 
\text{if $\sqrt{\frac{b(1-\theta_{a,b})}{a\theta_{a,b}}} \leq \vert z \vert < \sqrt{\frac{b(1-\vartheta_{a,b})}{a\vartheta_{a,b}}}$}. \\
\end{cases}
\]
and
\[
r_2(z) = \begin{cases}
-(1 - \vartheta_{a,b}) \log \vert z \vert^2 + \log (a \vert z \vert^2 + b) & 
\text{if $\sqrt{\frac{b(1-\theta_{a,b})}{a\theta_{a,b}}} < \vert z \vert \leq \sqrt{\frac{b(1-\vartheta_{a,b})}{a\vartheta_{a,b}}}$}, \\
0 & \text{if $\vert z \vert > \sqrt{\frac{b(1-\vartheta_{a,b})}{a\vartheta_{a,b}}}$}.
\end{cases}
\]
In order to see that $p_{a,b}$ is a $P_{a,b}$-Green function of $(\Tpsh \cap C^0)$-type on $\PP^1(\CC)$,
it is sufficient to check that $r_1$ and $r_2$ are  continuous and subharmonic on each area.
Let us see that $r_1$ is continuous and subharmonic. If $\theta_{a,b} = 1$, then the assertion is obvious, so that we may assume that
$\theta_{a,b} < 1$.
First of all, note that $r_1(z) = 0$ if $\vert z \vert = \sqrt{\frac{b(1-\theta_{a,b})}{a\theta_{a,b}}}$,  and hence
$r_1$ is continuous. It is obvious that $r_1$ is subharmonic on 
\[
\left\{ z \in \CC \ \left| \ \vert z \vert <\sqrt{\frac{b(1-\theta_{a,b})}{a\theta_{a,b}}}\right\}\right. \cup
\left\{ z \in \CC \ \left| \   \sqrt{\frac{b(1-\theta_{a,b})}{a\theta_{a,b}}} <  \vert z \vert < \sqrt{\frac{b(1-\vartheta_{a,b})}{a\vartheta_{a,b}}} \right\}\right..
\]
Note that
\[
\log(a x + b) - (1-\theta_{a,b})\log(x) \geq 0\quad(\forall x > 0)
\]
and the equality holds if and only if $x =\frac{b(1-\theta_{a,b})}{a\theta_{a,b}}$.
Thus $r_1 \geq 0$. Therefore, if $\vert z \vert = \sqrt{\frac{b(1-\theta_{a,b})}{a\theta_{a,b}}}$, then
\[
r_1(z) = 0 \leq \frac{1}{2\pi} \int_0^{2\pi} r_1(z + \epsilon e^{\sqrt{-1}t}) dt
\]
for a small positive real number $\epsilon$, and hence $r_1$ is subharmonic.
In a similar way, we can check that $r_2$ is continuous and subharmonic.

Next let us see that $\overline{P}_{a,b}$ is nef.
As $r_1(0) = 0$ and $r_2(\infty) = 0$, we have 
\[
\adeg(\rest{\overline{P}_{a,b}}{C_0}) = \adeg(\rest{\overline{P}_{a,b}}{C_{\infty}}) = 0.
\]
Note that 
\[
\overline{P}_{a,b} - \theta_{a,b} \widehat{(z)} = ((\theta_{a,b} - \vartheta_{a,b})C_{\infty}, p_{a,b}(z) + \theta_{a,b}\log\vert z \vert^2)
\]
and
\[
p_{a,b}(z) +  \theta_{a,b}\log\vert z \vert^2
= 
\begin{cases}
 r_1 & \text{if $\vert z \vert \leq \sqrt{\frac{b(1-\vartheta_{a,b})}{a\vartheta_{a,b}}}$}, \\
(\theta_{a,b} - \vartheta_{a,b}) \log \vert z \vert^2 & \text{if $\vert z \vert > \sqrt{\frac{b(1-\vartheta_{a,b})}{a\vartheta_{a,b}}}$}.
\end{cases}
\]
Therefore, $p_{a,b}(z) +  \theta_{a,b}\log\vert z \vert^2 \geq 0$ on $\PP^1(\CC)$, which means that
$\overline{P}_{a,b} - \theta_{a,b} \widehat{(z)}$ is effective.
Let $C$ be a $1$-dimensional closed integral subscheme of $\PP^1_{\ZZ}$ with
$C \not= C_0, C_{\infty}$.
Then
\[
\adeg(\rest{\overline{P}_{a,b}}{C}) = \adeg(\rest{((\theta_{a,b} - \vartheta_{a,b})C_{\infty}, p_{a,b} + \theta_{a,b}\log\vert z \vert^2)}{C}) \geq 0.
\]

Let us choose a positive number $m$ such that $\aH(\PP^1_{\ZZ}, m\overline{D}_{a,b}) \not= \{ 0 \}$.
By using Proposition~\ref{prop:criterion:QQ:effective}, we have
$\mu_{C_0}(m\overline{D}_{a,b}) = m - m\theta_{a,b}$ and $\mu_{C_{\infty}}(m\overline{D}_{a,b}) = m\vartheta_{a,b}$
because
\[
\nu_{C_0}(nm\overline{D}_{a,b}) = nm -  \lfloor  nm \theta_{a,b} \rfloor\quad\text{and}\quad
\nu_{C_{\infty}}(nm\overline{D}_{a,b}) = \lceil nm\vartheta_{a, b} \rceil
\]
(for the definitions of $\nu$ and $\mu$, see \cite[SubSection~6.5]{MoArZariski}).
Thus
the positive part of $\overline{D}_{a,b}$
can be written by a form 
$(P_{a,b}, q)$,
where
$q$ is a $P_{a,b}$-Green function of $(\Tpsh \cap C^{0})$-type on $\PP^1(\CC)$ (cf. \cite[Claim~9.3.5.1 and Proposition~9.3.1]{MoArZariski}).
Note that $\overline{P}_{a,b}$ is nef and $\overline{P}_{a,b} \leq \overline{D}_{a,b}$, so that 
\[
p_{a, b}(z) \leq q(z) \leq -\log \vert z \vert^2 + \log(a\vert z \vert^2 + b).
\]
We choose a continuous function $u$ such that $p_{a,b} + u = q$.
Then $u(z) = 0$ on 
\[
\sqrt{\frac{b(1-\theta_{a,b})}{a\theta_{a,b}}} \leq \vert z \vert \leq \sqrt{\frac{b(1-\vartheta_{a,b})}{a\vartheta_{a,b}}}.
\]
Moreover, since $q(z) = -\theta_{a,b} \log \vert z \vert^2 + u(z)$ on 
$\vert z \vert \leq \sqrt{\frac{b(1-\theta_{a,b})}{a\theta_{a,b}}}$,
$u$ is subharmonic on $\vert z \vert \leq \sqrt{\frac{b(1-\theta_{a,b})}{a\theta_{a,b}}}$.
On the other hand,
$u(0) = 0$ because 
\[
\adeg (\rest{(P_{a,b}, q)}{C_0}) = u(0) = 0.
\]
Therefore, $u = 0$ on $\vert z \vert \leq \sqrt{\frac{b(1-\theta_{a,b})}{a\theta_{a,b}}}$
by the maximal principle.
In a similar way, we can see that $u = 0$ on $\vert z \vert \geq \sqrt{\frac{b(1-\vartheta_{a,b})}{a\vartheta_{a,b}}}$.

\medskip
Next we consider the case where $a+ b = 1$. 
As $\log$ is a concave function,
\[
(1-b) \log \vert z \vert^2 \leq \log((1-b)\vert z \vert^2 + b)
\]
on $\PP^1(\CC)$. Thus 
$b\widehat{(z)} \leq \overline{D}_{a,b}$, and hence the Zariski decomposition of $\overline{D}_{a,b}$ exists by
\cite[Theorem~9.2.1]{MoArZariski}.
Let $\overline{P}$ be the positive part of $\overline{D}_{a,b}$. Then $b\widehat{(z)} \leq \overline{P}$.

Let us consider the converse inequality.
Let $t$ be a real number with $t > 1$.
Since $\overline{P} \leq \overline{D}_{a,b} \leq \overline{D}_{ta,tb}$,
we have $\overline{P} \leq \overline{P}_{ta,tb}$ because $\overline{P}_{ta,tb}$ is the positive part of $\overline{D}_{ta,tb}$
by the previous observation.  Since $\varphi_{ta,tb} = \varphi_{a,b} + \log (t)$,
we have $\lim_{t\to 1} \vartheta_{ta,tb} = \lim_{t\to 1} \theta_{ta,tb} = b$.
Therefore, we can see
\[
\lim_{t\to 1} \overline{P}_{ta, tb} = \overline{P}_{a,b} = b\widehat{(z)}.
\]
Thus $\overline{P} \leq b\widehat{(z)}$.

\medskip
Finally we consider the case where $a+ b < 1$.
Then, by Theorem~\ref{thm:positivity:D:a:b}, $\overline{D}_{a,b}$ is not pseudo-effective.
Thus the Zariski decomposition does not exist by \cite[Proposition~9.3.2]{MoArZariski}.
\end{proof}

\end{document}